\documentclass[a4paper,oneside,11pt]{amsart}

\reversemarginpar

\usepackage[colorlinks,hyperindex,linkcolor=blue,urlcolor=black,pdftitle={A sufficient condition for the similarity of a
polynomially bounded operator to a contraction},pdfauthor={Maria F.Gamal'}]
{hyperref}

\usepackage{amsmath}
\usepackage{amsfonts}
\usepackage{amssymb}
\usepackage{amsthm}

\usepackage[english]{babel}
\usepackage{enumerate}

\usepackage{graphicx}
\usepackage{color}

\usepackage[abbrev]{amsrefs}

\numberwithin{equation}{section}

\theoremstyle{plain}
\newtheorem{theorem}{Theorem}[section]
\newtheorem{lemma}[theorem]{Lemma}

\newtheorem{proposition}[theorem]{Proposition}
\newtheorem{theoremcite}{Theorem}

\theoremstyle{definition}
\newtheorem{remark}[theorem]{Remark}
\newtheorem{example}[theorem]{Example}

\begin{document}

\title[A sufficient condition for the similarity]{A sufficient condition for the similarity of a
polynomially bounded operator to a contraction}

\author{Maria F. Gamal'}
\address{
 St. Petersburg Branch\\ V. A. Steklov Institute 
of Mathematics\\
 Russian Academy of Sciences\\ Fontanka 27, St. Petersburg\\ 
191023, Russia  
}
\email{gamal@pdmi.ras.ru}




\begin{abstract}
 Let $T$ be a  polynomially bounded operator, and let $\mathcal M$ be its invariant subspace. 
Suppose that $P_{\mathcal M^\perp}T|_{\mathcal M^\perp}$ is similar to a contraction, while $\theta(T|_{\mathcal M})=0$, where 
$\theta$ is a finite product of Blaschke products with simple zeros satisfying the Carleson interpolating condition 
(a Carleson--Newman Blaschke product). 
Then $T$ is similar to a contraction. It is mentioned that Le Merdy's example shows that the assumption of  polynomially 
boundedness cannot be replaced by the assumption of  power boundedness.  
  \end{abstract}

\maketitle

\section{Introduction}

 Let $\mathcal H$ be a (complex, separable) Hilbert space, and let $\mathcal B(\mathcal H)$ be the space of (linear, bounded) 
operators acting on $\mathcal H$.  An operator $R\in\mathcal B(\mathcal H)$ is called {\it power bounded}, if 
$\displaystyle{\sup_{n\geq 1}\|R^n\| < \infty}$. An operator $R$ is called {\it polynomially bounded}, 
if there exists a constant $M>0$ such that 
$$\|p(R)\|\leq M \sup\{|p(z)|: \ |z|\leq 1\} \  \text{ for every polynomial } \ p .$$ An operator $T$ is called a {\it contraction} if $\|T\|\leq 1$. 
Clearly, a polynomially bounded operator is power bounded. A contraction is polynomially bounded with the constant $M=1$
(von Neumann inequality; see, for example, {\cite[Proposition I.8.3]{sfbk}}). 

Let $\mathcal H$, $\mathcal K$  be two Hilbert spaces, and let  $\mathcal B(\mathcal H,\mathcal K)$ be the space of 
 (linear, bounded) operators acting from $\mathcal H$ to $\mathcal K$.
Two operators $R\in\mathcal B(\mathcal H)$ and $T\in\mathcal B(\mathcal K)$ are called {\it similar}, if there exists an {\it invertible} 
operator $X\in\mathcal B(\mathcal H,\mathcal K)$  
such that $XR = TX$,  in notation: $R\approx T$. (An operator $X$ is called invertible, if it has a {\it bounded}
 inverse $X^{-1}$.) Clearly, the power boundedness and the polynomially boundedness are preserved under similarity. 

The question whether each power bounded operator is similar to a contraction was posed by Sz.-Nagy \cite{sznagy}. A negative answer 
on this question was given by Foguel \cite{foguel}, where a power bounded but not polynomially bounded operator was constructed 
(see also \cite{halmos64} and \cite{lebow}). Therefore, in \cite{halmos70} the question 
whether each polynomially bounded operator is similar to a contraction 
was posed. A negative answer on this question was given by Pisier \cite{pisier}, see also \cite{dp}. 

The example of a power bounded operator which is not polynomially bounded from \cite{foguel} and the example of a 
polynomially bounded operator which is not similar to a contraction from \cite{pisier} have the form 
\begin{equation}\label{1.1}\begin{pmatrix} T_0 & \ast \\ \mathbb O & T_1\end{pmatrix}, \end{equation}
where $T_0$ and $T_1$ are contractions. The question on additional conditions which garantee that the operator of 
the form \eqref{1.1} is similar to a contraction was considered in \cite{cassier} (among other questions) and \cite{badea}.  
The following theorem is a particular case of {\cite[Corollary 4.2]{cassier}}.

\begin{theoremcite}\cite{cassier}\label{tha}  Assume that an operator $V$ is similar to an isometry,  an operator $T$ is similar  to a contraction, and 
$R$ is a power bounded operator of the form 
\begin{equation}\label{1.2} R=\begin{pmatrix} V & \ast \\ \mathbb O & T\end{pmatrix}. \end{equation} 
Then $R$ is similar to a contraction.\end{theoremcite}

Clearly,  an operator  $R$ from Theorem \ref{tha} is not necessary similar to $V\oplus T$. For example, if $R$ is  
a unilateral shift of  multiplicity 1, then $R$ can be represented in the form \eqref{1.2} and $R$ is not similar to  $V\oplus T$. 
On the other hand, the following well-known theorems take place. 

\begin{theoremcite} \cite{badea}\label{thb} Assume that  an operator  $U$ is similar to a unitary,  and 
$R$ is a  power bounded operator of the form $$R = \begin{pmatrix} U & \ast \\ \mathbb O & T\end{pmatrix}. $$ 
Then $R\approx U\oplus T$. \end{theoremcite}

\begin{theoremcite} \cite{clark}\label{thc} Assume that $V\in\mathcal B(\mathcal K)$ is similar to an isometry, $T\in\mathcal B(\mathcal H)$ is similar  to 
a contraction, $A\in\mathcal B(\mathcal K, \mathcal H)$, and $R\in\mathcal B(\mathcal H\oplus\mathcal K)$ is of the form 
$$R = \begin{pmatrix} T & A \\ \mathbb O & V\end{pmatrix}. $$ 
Then $R$ is similar to a contraction if and only if there exists $Y\in\mathcal B(\mathcal K, \mathcal H)$ such that 
$A = TY-YV$, and then $R\approx T\oplus V$, because 
$$ \begin{pmatrix} I & Y \\ \mathbb O & I\end{pmatrix}\begin{pmatrix} T & A \\ \mathbb O & V\end{pmatrix} = 
\begin{pmatrix} T & \mathbb O \\ \mathbb O & V\end{pmatrix}\begin{pmatrix} I & Y \\ \mathbb O & I\end{pmatrix}.$$\end{theoremcite}

For the proof of Theorem \ref{thb}, see {\cite[Corollary 2.2]{badea}} 
and references therein, for the proof of Theorem \ref{thc}, see, for example,  {\cite[Section 2]{clark}}.

We mention here the following simple fact.

\begin{lemma}\label{lem1.1.} Assume that $T\in\mathcal B(\mathcal H)$ is power bounded, and $\mathcal H= \mathcal M\vee\mathcal N$, where $\mathcal M$ and $\mathcal N$ 
are (closed) subspaces of $\mathcal H$ such that $\mathcal M$ and $\mathcal N$ are invariant for $T$, $T|_{\mathcal M}$ is similar to an isometry, 
and $T|_{\mathcal N}$ is stable, that is, $\|T^nx\|\to_n0$ for every $x\in\mathcal N$. Then $T\approx T|_{\mathcal M}\oplus T|_{\mathcal N}$.\end{lemma}

\begin{proof} Set $C=\sup_{n\geq 0}\|T^n\|$. Since $T|_{\mathcal M}$ is similar to an isometry, there exists $c>0$ such that  
$\|T^nx\|\geq c\|x\|$ for every $x\in \mathcal M$, $n\geq 0$. Clearly, 
$$\|x+y\|\leq\|x\|+\|y\|\leq \sqrt 2 (\|x\|^2+\|y\|^2)^{1/2} =\sqrt 2 \|x\oplus y\|.$$
Define the operator $X\in\mathcal B(\mathcal M\oplus \mathcal N,\mathcal H)$ by the formula 
$X(x\oplus y)= x+y$, $x\in \mathcal M$, $y\in \mathcal N$.  Let $x\in \mathcal M$, and let $y\in \mathcal N$. We have
$$C\|x+y\|\geq \|T^n(x+y)\|\geq \|T^n x\|-\|T^n y\|\geq c\|x\|-\|T^n y\|\to_n c\|x\|.$$
Therefore, $\|x\|\leq \frac{C}{c}\|x+y\|$ for every $x\in \mathcal M$, $y\in \mathcal N$.  From the estimates
\begin{align*}\|x\|^2+\|y\|^2 & \leq (\|x+y\|+\|x\|)^2+\|x\|^2 \\ & \leq \Bigl(1+ \frac{C}{c}\Bigr)^2\|x+y\|^2  + \Bigl(\frac{C}{c}\Bigr)^2\|x+y\|^2 \\ &
  = \Bigl (1+ 2\frac{C}{c}  + 2\Bigl(\frac{C}{c}\Bigr)^2\Bigr)\|x+y\|^2 \end{align*}
we conclude that $X$ is invertible and $\|X^{-1}\|\leq \bigl(1+ 2\frac{C}{c}  + 2\bigl(\frac{C}{c}\bigr)^2\bigr)^{1/2}$. 
The relation $X(T|_{\mathcal M}\oplus T|_{\mathcal N}) = TX$ is evident. \end{proof} 

Let $R\in\mathcal B(\mathcal H)$ be a polynomially bounded operator. 
Then there exist two invariant subspaces $\mathcal H_a$ and $\mathcal H_s$ of $R$  
such that $\mathcal H = \mathcal H_a \dotplus \mathcal H_s$, $R|_{\mathcal H_a}$ is {\it  absolutely continuous (a.c.)}
and $R|_{\mathcal H_s}$ is singular. Any polynomially bounded singular operator 
is similar to a singular unitary operator. 
For an absolutely continuous polynomially bounded operator $R$ 
the  $H^\infty$-calculus is well defined, 
that is, for every $\varphi\in H^\infty$ the operator $\varphi(R)$ is defined such that 
$\|\varphi(R)\|\leq M\|\varphi\|_\infty$, 
where $M$ is a constant from the definition of polynomially boundedness of $R$. 
(Here $H^\infty$ is the algebra of bounded analytic functions in the open unit disk $\mathbb D$.)
For $\varphi\in H^\infty$ set $\widetilde\varphi(z)=\overline{\varphi(\overline z)}$, $z\in\mathbb D$. Then 
$\widetilde\varphi\in H^\infty$ and 
\begin{equation}\label{1.3} \widetilde\varphi(R^\ast) = \varphi(R)^\ast \end{equation} for every a.c. polynomially bounded operator $R$. 
For these facts on  polynomially bounded operators see \cite{mlak} or \cite{kerchy}.  
An a.c. polynomially bounded operator $R$ is called {\it of class} $C_0$, if there exists $\varphi\in H^\infty$ such that 
$\varphi\not\equiv 0$ and $\varphi(R)=\mathbb O$ \cite{bercpr}. Since an operator $R$ of class $C_0$ is quasisimilar to a contraction 
(see \cite{bercpr}), it follows from {\cite[III.4.4]{sfbk}} that there exists an inner function $\theta$ such that $\theta(R)=\mathbb O$. 

Let $\theta$ be an inner function, and let  $\theta$ have the following property: 
\begin{equation}\label{1.4}\begin{gathered} \text{if } T \text { is an  a.c. polynomially bounded operator such that } \theta(T)=\mathbb O, \\
\text{ then } T \text{ is similar to a contraction.} \end{gathered}\end{equation}
It is clear from \eqref{1.3} that $\theta$ has the property \eqref{1.4} if and only if $\widetilde\theta$  has the property \eqref{1.4}.

The following theorem is the main result of the present paper.  

\begin{theorem}\label{th1.2} Assume that  $T_0$ is an  a.c. polynomially bounded operator 
and there exists an inner function $\theta$ satisfying \eqref{1.4}  such that $\theta(T_0)=\mathbb O$,
an operator  $T_1$ is similar to a contraction, and 
$$R=\begin{pmatrix} T_0 & \ast\\  \mathbb O & T_1 \end{pmatrix}$$ is polynomially bounded. Then $R$ is similar to a contraction. \end{theorem}

This theorem has a simple corollary.

\begin{theorem}\label{th1.3} Assume that  $\mathcal H_0$ and $\mathcal H_1$ are Hilbert spaces, $T_0\in \mathcal B(\mathcal H_0)$,
$T_1\in \mathcal B(\mathcal H_1)$, and $A\in \mathcal B(\mathcal H_0,\mathcal H_1)$ are such that $T_0$ is of class $C_0$ 
and there exists an inner function $\theta$ satisfying \eqref{1.4} such that $\theta(T_0)=\mathbb O$,
 $T_1$ is similar to a contraction, and 
$$R=\begin{pmatrix} T_1 & A\\  \mathbb O & T_0 \end{pmatrix}$$ is polynomially bounded. Then $R$ is similar to a contraction.\end{theorem}

\begin{proof} $R^\ast$ has the following matrix representation on the space $\mathcal H_0\oplus\mathcal H_1$:
$$R^\ast=\begin{pmatrix} T_0^\ast & A^\ast\\  \mathbb O & T_1^\ast \end{pmatrix}.$$ Since $\widetilde\theta(T_0^\ast)=
\theta(T_0)^\ast=\mathbb O$ and $\widetilde\theta
$ satisfies (1.4), Theorem 1.2 can be applied to $R^\ast$. \end{proof}

It is well known that if $\theta$ is a Blaschke product with simple zeros $\{\lambda_n\}_n\subset\mathbb D$ 
satisfying the Carleson interpolating  condition (the definition is recalled in Sec. 3), and $T$ is an  a.c. polynomially bounded operator 
such that $\theta(T)=\mathbb O$, then $$T\approx\bigoplus_n \lambda_nI_{\ker(T-\lambda_nI)}$$
(see, for example, {\cite[Proof of Theorem 2.1]{cassierest}} or {\cite[Lemma 2.3]{vitsejfa}}, see also Lemma \ref{lem3.1} of the present paper and references before it). 
Therefore, $T$ is similar to a contraction. Thus, a  Blaschke product with simple zeros  
satisfying the Carleson condition has the property \eqref{1.4}. Using results from \cite{bour} and  \cite{vas}, it is easy to see that 
a finite product of Blaschke products with simple zeros  
satisfying the Carleson condition (a Carleson--Newman product) has the property \eqref{1.4}, see Theorem \ref{th3.2} below. On the other hand, by using the sequence 
of finite-dimensional operators that are uniformly polynomially bounded, but not uniformly complete polynomially bounded 
(see \cite{pisier}, \cite{dp} for the definition of the complete polynomially boundedness and the construction of such sequence), 
it is easy to construct an  a.c. polynomially bounded operator $T$ such that  $\theta(T)=\mathbb O$ 
for some  Blaschke product $\theta$ and $T$ is not similar to a contraction, see the last paragraph of Section 5 in \cite{g1}, 
see also {\cite[Theorem 7.1]{g1}}. Thus, there exist Blaschke products 
which are not satisfy \eqref{1.4}. Furthermore, it is shown in \cite{g2} that there exist 
a.c. polynomially bounded operators $T$ such that  $\theta(T)=\mathbb O$ and $T$ is not similar to a contraction, where 
$\theta(z)=\exp(a\frac{z+1}{z-1})$, $a>0$, $z\in\mathbb D$.
Author does not know whether  there exist inner functions satisfying \eqref{1.4} except Carleson--Newman products.

The paper is organized as follows. In Sec. 2, Theorem \ref{th1.2} is proved. In Sec. 3, it is shown that Carleson--Newman products   
 satisfy the condition \eqref{1.4}. 
In Sec. 4, Le Merdy's example \cite{lemerdy} is used to show that Theorem \ref{th1.2} can not be generalized to power bounded operators. 

\section{Proof of Theorem \ref{th1.2} }

The following lemmas are simple and well known.

\begin{lemma}\label{lem2.1} Assume that  $R\in\mathcal B(\mathcal H)$ is a polynomially bounded operator, and 
$\mathcal H = \mathcal H_a \dotplus \mathcal H_s$, where $\mathcal H_a$ and $\mathcal H_s$
are invariant subspaces for $R$ such that $R|_{\mathcal H_a}$ is  a.c. and $R|_{\mathcal H_s}$ is singular. 
Suppose that $\mathcal M$ is an invariant subspace for $R$ such that $R|_{\mathcal M}$ is  a.c.. Then $\mathcal M\subset \mathcal H_a$.\end{lemma}

\begin{proof} Denote by $Q$ the skew projection onto $\mathcal H_s$ 
parallel to  $\mathcal H_a$, that is, $Q\in\mathcal B(\mathcal H)$,  
$Q|_{\mathcal H_s} = I_{\mathcal H_s}$ and $Q|_{\mathcal H_a} = \mathbb O$. Then 
$Q|_{\mathcal M} R|_{\mathcal M} = R|_{\mathcal H_s} Q|_{\mathcal M}$.
Since  $R|_{\mathcal M}$ is  a.c.  and $R|_{\mathcal H_s}$ is singular,
we have $Q|_{\mathcal M} =\mathbb O$ (see \cite{mlak} or {\cite[ Proposition 15]{kerchy}}). 
 It means that $\mathcal M\subset \mathcal H_a$. \end{proof}

\begin{lemma}\label{lem2.2} Assume that  $R$ is a polynomially bounded operator, $\mathcal M$ is its invariant subspace, and $R|_{\mathcal M}$ 
and $P_{\mathcal M^\perp} R|_{\mathcal M^\perp}$ are a.c.. Then $R$ is a.c..\end{lemma}
 
\begin{proof} Denote by $\mathcal H$ the space on which $R$ acts. Let $\mathcal H_a$ and $\mathcal H_s$ be as in Lemma \ref{lem2.1}. Then 
$\mathcal M\subset \mathcal H_a$, consequently, $\mathcal M^\perp\supset\mathcal H_a^\perp$. Clearly, $\mathcal M^\perp$ and $\mathcal H_a^\perp$ 
are invariant subspaces for $R^\ast$, $R^\ast|_{\mathcal M^\perp} = (P_{\mathcal M^\perp} R|_{\mathcal M^\perp})^\ast$ and 
$R^\ast|_{\mathcal H_a^\perp} = (P_{\mathcal H_a^\perp} R|_{\mathcal H_a^\perp})^\ast$. 
The operator $P_{\mathcal H_a^\perp}|_{\mathcal H_s}$ realizes the relation $P_{\mathcal H_a^\perp} R|_{\mathcal H_a^\perp}\approx R|_{\mathcal H_s}$, therefore, $P_{\mathcal H_a^\perp} R|_{\mathcal H_a^\perp}$ 
is singular (\cite{mlak} or  {\cite[ Proposition 16]{kerchy}}). Thus, $R^\ast|_{\mathcal H_a^\perp}$ is singular  (\cite{mlak} or {\cite[Proposition 14]{kerchy}}). Since 
$R^\ast|_{\mathcal M^\perp}$ is a.c., we conclude that ${\mathcal H_a^\perp} = \{0\}$. \end{proof}

The following theorem is based on Theorem \ref{tha}, which is a particular case of {\cite[Corollary 4.2]{cassier}}.

\begin{theorem}\label{th2.3} Assume that $\mathcal H_0$ and $\mathcal K$ are Hilbert spaces, $T_0\in \mathcal B(\mathcal H_0)$ is a polynomially bounded operator 
 of class $C_0$ 
and there exists an inner function $\theta$ satisfying $(1.4)$ such that $\theta(T_0)=\mathbb O$, 
$V\in \mathcal B(\mathcal K)$ is an a.c. isometry, and $A\in \mathcal B(\mathcal K,\mathcal H_0)$ is such that 
\begin{equation}\label{2.1}R=\begin{pmatrix} T_0 & A\\  \mathbb O & V \end{pmatrix}\end{equation}
 is polynomially bounded. Then $R$ is similar to a contraction.\end{theorem}

\begin{proof} By Lemma \ref{lem2.2}, $R$ is a.c., therefore, $\theta(R)$  is well-defined.  Put 
$\mathcal M = \operatorname{clos}\theta(R)(\mathcal H_0\oplus\mathcal K)$. Clearly, $R\mathcal M\subset\mathcal M$ and 
$\theta(P_{\mathcal M^\perp}R|_{\mathcal M^\perp})=P_{\mathcal M^\perp}\theta(R)|_{\mathcal M^\perp}= \mathbb O$.
Since $\theta$ satisfies \eqref{1.4}, $P_{\mathcal M^\perp}R|_{\mathcal M^\perp}$ is similar to a contraction.

We will show that $R|_{\mathcal M}\approx V$. Let  $A_\theta\in \mathcal B(\mathcal K,\mathcal H_0)$ be  the operator such that
\begin{equation}\label{2.2} \theta(R)=\begin{pmatrix} \mathbb O & A_\theta\\  \mathbb O & \theta(V) \end{pmatrix}.\end{equation}
From the equation $R\theta(R)=\theta(R)R$ written in the matrix form we conclude that 
\begin{equation}\label{2.3} T_0 A_\theta + A \theta(V) = A_\theta V. \end{equation}
Put $X=\theta(R)|_{\mathcal K}$. It is clear from \eqref{2.2} that  $\mathcal M = \operatorname{clos}X\mathcal K$.
Thus,  $X\in\mathcal B(\mathcal K,\mathcal M)$, and $X$ acts by the formula $Xx = A_\theta x\oplus \theta(V)x$, $x\in\mathcal K$. 
We have $$\|Xx\|^2=\|A_\theta x\|^2 + \| \theta(V)x\|^2\geq \| \theta(V)x\|^2 = \|x\|^2$$ 
(because $\theta$ is inner and $V$ is an a.c. isometry). Therefore, $X$ is invertible. It follows from \eqref{2.3} that 
$R|_{\mathcal M}X = XV$. Thus, $X$  realizes the relation $R|_{\mathcal M}\approx V$.

We have obtained that $R$ has an invariant subspace $\mathcal M$ such that $R|_{\mathcal M}$ is similar to an isometry and 
$P_{\mathcal M^\perp}R|_{\mathcal M^\perp}$ is similar to a contraction. By Theorem \ref{tha}, $R$ is similar to a contraction. \end{proof}

\begin{remark}\label{rem2.4} 
Assume that $R$ has the form \eqref{2.1}, where  $V$ is an isometry. 
By Theorem \ref{thc}, 
$R$ is similar to  a contraction if and only if $T_0$ is similar to a contraction and there exists an 
operator $Y$ such that $A=T_0Y-YV$, and then 
\begin{equation}\label{2.4}\begin{pmatrix} I & Y\\  \mathbb O & I \end{pmatrix} R =\begin{pmatrix}  T_0 & \mathbb O\\  \mathbb O & V \end{pmatrix} 
\begin{pmatrix}  I & Y\\  \mathbb O & I \end{pmatrix}. \end{equation} 
 Now suppose that $R$ satisfies the conditions of Theorem \ref{th2.3}. Then 
there exists $Y\in \mathcal B(\mathcal K,\mathcal H_0)$ such that $A=T_0Y-YV$. Let $A_\theta$ be defined in \eqref{2.2}. 
It is easy to see from \eqref{2.4} and \eqref{2.2} that $Y\theta(V)=-A_\theta$, that is, $Y$ can be defined on $\theta(V)\mathcal K$ 
by the only way. Author does not know how  $Y$ can be defined on  $\mathcal K\ominus\theta(V)\mathcal K$, see also Example \ref{ex4.1} below. \end{remark}

\begin{proposition}\label{prop2.5} Assume that $\mathcal H_0$, $\mathcal H_1$, $\mathcal K_1$ are Hilbert spaces,  $T_0\in \mathcal B(\mathcal H_0)$,
$T_1\in \mathcal B(\mathcal H_1)$, $V_1\in \mathcal B(\mathcal K_1)$, $K\in \mathcal B(\mathcal H_1,\mathcal K_1)$, and $A\in \mathcal B(\mathcal H_1,\mathcal H_0)$.
Set $\mathcal H =\mathcal H_0\oplus\mathcal H_1$, $\mathcal K=\mathcal K_1\oplus\mathcal H_1$, 
$$V = \begin{pmatrix} V_1 & K \\  \mathbb O & T_1 \end{pmatrix}\in\mathcal B(\mathcal K), \ \ \ \ 
R = \begin{pmatrix} T_0 & A\\  \mathbb O & T_1 \end{pmatrix}\in\mathcal B(\mathcal H), $$
and
\begin{equation}\label{2.5} R_0=\begin{pmatrix} T_0 & \mathbb O &  A\\  \mathbb O & V_1 & K \\ \mathbb O & \mathbb O & T_1 \end{pmatrix}
\in\mathcal B(\mathcal H_0\oplus\mathcal K). \end{equation} 
Then $\mathcal K_1$ is an invariant subspace of $R_0$, and $R_0$ has the following form with respect to the representation 
of the space $\mathcal H_0\oplus\mathcal K$ as $\mathcal K_1\oplus\mathcal H$:
\begin{equation}\label{2.6} R_0 = \begin{pmatrix} V_1 & \ast \\  \mathbb O & R \end{pmatrix}\in\mathcal B(\mathcal K_1\oplus\mathcal H). \end{equation}

$1)$ Assume that $V$ is power bounded (polynomially bounded). Then $R$ is power bounded (polynomially bounded) if and only if 
$R_0$ is power bounded (polynomially bounded).

$2)$ Assume that $V$ is an isometry. Then $R$ is similar to a contraction if and only if 
$R_0$ is similar to a contraction.\end{proposition}

\begin{proof} 1) For a natural number $n$ let $K_n\in \mathcal B(\mathcal H_1,\mathcal K_1)$ and $A_n\in \mathcal B(\mathcal H_1,\mathcal H_0)$
be operators such that 
$$V^n = \begin{pmatrix} V_1^n & K_n \\ \mathbb O & T_1^n \end{pmatrix} \ \ \text{ and } \ \ 
R^n = \begin{pmatrix} T_0^n & A_n\\  \mathbb O & T_1^n \end{pmatrix}. $$
Easy calculation shows that 
$$R_0^n=\begin{pmatrix} T_0^n & \mathbb O &  A_n \\  \mathbb O & V_1^n & K_n \\ \mathbb O & \mathbb O & T_1^n \end{pmatrix}.$$
Clearly, $V$ is power bounded if and only if the norms of operators from matrix representation of $V^n$
are bounded uniformly by $n$, and the same is true for $R$ and $R_0$. Now the statements on power boundedness follows.  
The statement on polynomially boundedness is proved similarly. 

2)  Assume that $V$ is an isometry, then $V_1$ is an isometry, because $V_1=V|_{\mathcal K_1}$. If $R$ is similar to a contraction, 
then $R_0$ is similar to a contraction by Theorem \ref{tha} and \eqref{2.6}. 
If $R_0$ is similar to a contraction, then $R$ is similar to a contraction, because $R$ is the compression of $R_0$ on
its coinvariant subspace. \end{proof}

\smallskip

{\bf Proof of Theorem \ref{th1.2}.} Without loss of generality, we can suppose that $T_1$ is a completely nonunitary contraction. Indeed, 
let $T_2$ be a contraction and let $Y$ be an invertible operator such that $YT_1Y^{-1}=T_2$. Then 
$$R_1:=\begin{pmatrix} I & \mathbb O \\  \mathbb O & Y \end{pmatrix}  R  \begin{pmatrix} I & \mathbb O \\  \mathbb O & Y^{-1}\end{pmatrix}
 = \begin{pmatrix} T_0 & \ast \\  \mathbb O & T_2 \end{pmatrix}.$$
Furthermore, $T_2=T_3\oplus U$, where $T_3$ is a completely nonunitary contraction and 
$U$ is a unitary operator. Then 
$$R_1=\begin{pmatrix} T_0 & A_1 &  A_2 \\  \mathbb O & T_3 & \mathbb O \\ \mathbb O & \mathbb O & U \end{pmatrix},$$
where $A_1$ and $A_2$ are appropriate operators.
By Theorem \ref{thb} (applied to $R_1^\ast$), $R_1$ is similar to a contraction if and only if 
$\begin{pmatrix} T_0 & A_1 \\  \mathbb O & T_3 \end{pmatrix}$ 
is similar to a contraction. 

Thus, suppose that $R=\begin{pmatrix} T_0 & A \\ \mathbb O & T_1 \end{pmatrix}$, where $\theta(T_0)=\mathbb O$ and 
$T_1$ is a completely nonunitary contraction. Denote by $V$ the minimal isometric dilation of $T_1$.
Since $T_1$ is  completely nonunitary, $V$ is a.c. (see {\cite[Theorems I.4.1 and II.6.4]{sfbk}}). 
We have $V=\begin{pmatrix} V_1 & K \\  \mathbb O & T_1 \end{pmatrix}$ 
for some operators $V_1$ and $K$. Define $R_0$ by \eqref{2.5}. By Proposition \ref{prop2.5} (1), $R_0$ is polynomially bounded. 
By Theorem 2.3, $R_0$ is similar to a contraction. By Proposition  \ref{prop2.5} (2), $R$ is similar to a contraction. 

\section{On functions satisfying the condition \eqref{1.4}}
 
In this section, we show that inner functions satisfying the condition \eqref{1.4} exist (Theorem \ref{th3.2}). Also, we show that 
the finite product of functions satisfying the condition \eqref{1.4} satisfies the condition \eqref{1.4} (Proposition \ref{prop3.3}).

Recall the definitions. Let $\mathbb D$ be the open unit disc. For $\lambda\in\mathbb D$ set 
$b_\lambda(z)= \frac{|\lambda|}{\lambda}\frac{\lambda-z}{1-\overline\lambda z}$, $z\in\mathbb D$.
It is well known that if $\{\lambda_n\}_n\subset\mathbb D$ and $\sum_n(1-|\lambda_n|)<\infty$, then the product $B=\prod_nb_{\lambda_n}$
converges and is called a {\it Blaschke product}. Set $B_n=\prod_{k\neq n}b_{\lambda_k}$. The sequence $\{\lambda_n\}_n\subset\mathbb D$  
satisfies the {\it Carleson interpolating condition} (the Carleson condition for brevity), if $\inf_n|B_n(\lambda_n)|>0$. 
Assume that $\{\theta_n\}_n$ is a sequence of inner functions such that the product $\theta=\prod_n\theta_n$ converges. The sequence   
$\{\theta_n\}_n$ satisfies the {\it generalized Carleson condition}, if   there exists $\delta>0$ such that 
$|\theta(z)|\geq \delta\inf_{n\in\mathbb N} |\theta_n(z)|$ for every $z\in \mathbb D$. It is well known that the sequence 
$\{\lambda_n\}_n\subset\mathbb D$ satisfies the Carleson condition if and only if the sequence $\{b_{\lambda_n}\}_n$ 
satisfies  the generalized  Carleson condition, see, for example, {\cite[\S IX.3]{nik86}} or {\cite[Lemma II.C.3.2.18]{nik02}}.      

The following lemma is a straightforward consequence of {\cite[Sec. 3.4]{vas}}, see also {\cite[\S IX.2, IX.3]{nik86}} or 
{\cite[Theorems II.C.3.1.4 and II.C.3.2.14]{nik02}}. See also {\cite[Theorem 4.4]{clouatre}} for another proof. We sketch the proof of this lemma, because 
results in \cite{vas} are formulated for the compressions of the unilateral shift of multiplicity $1$ only.

\begin{lemma}\label{lem3.1} Assume that $\theta=\prod_{n\in\mathbb N}\theta_n$, where $\{\theta_n\}_{n\in\mathbb N}$ is a sequence of 
inner functions satisfying the generalized Carleson condition. Assume that $T$ is an a.c. polynomially 
bounded operator such that $\theta(T)=\mathbb O$. Then $T\approx \oplus_{n\in\mathbb N}T|_{\ker\theta_n(T)}$.\end{lemma}

\begin{proof} Denote by $\mathcal H$ the space on which $T$ acts, and set $\mathcal H_n=\ker\theta_n(T)$. 
Let $\{\mu_n\}_{n\in\mathbb N}\in\ell^\infty$. Then there exists $\varphi\in H^\infty$ such that $\varphi-\mu_n\in\theta_n H^\infty$
for all $n\in\mathbb N$ ({\cite[Theorem II.C.3.2.14]{nik02}}, or {\cite[Theorems 2.3 and 3.4]{vas}}, or {\cite[\S IX.2 and \S IX.3]{nik86}}).
 Clearly, $\varphi(T)$ is a bounded operator, and 
$\varphi(T)|_{\mathcal H_n}=\mu_nI_{\mathcal H_n}$ for all $n\in\mathbb N$. 
We obtain that for every $\{\mu_n\}_{n\in\mathbb N}\in\ell^\infty$ the operator
$$ \sum_{n\in\mathbb N}x_n\mapsto \sum_{n\in\mathbb N}\mu_n x_n, \ \ \ \text{ where }  \  x_n\in \mathcal H_n,$$ 
is bounded.
 By {\cite[Theorem II.C.3.1.4]{nik02}}, 
or {\cite[Theorem 2.1]{vas}}, or {\cite[\S VI.4]{nik86}}, the mapping 
$$X\colon\mathcal H\to\bigoplus_{n\in\mathbb N}\mathcal H_n, \ \ \  \sum_{n\in\mathbb N}x_n\mapsto \bigoplus_{n\in\mathbb N}x_n , 
\ \  x_n\in \mathcal H_n,$$ 
is bounded and invertible. Clearly, $XT=(\bigoplus_{n\in\mathbb N}T|_{\mathcal H_n})X$, that is, 
$T\approx \bigoplus_{n\in\mathbb N}T|_{\mathcal H_n}$. \end{proof}

\begin{theorem}\label{th3.2} Assume that $N$ is a natural number, $B_1, \ldots, B_N$ are Blaschke products with simple zeros satifying the 
Carleson condition,
and $\theta = B_1 \cdots  B_N$ (that is, $\theta$ is a Carleson--Newman product). Then $\theta$ satisfies \eqref{1.4}. \end{theorem}

\begin{proof} By {\cite[Theorem 5.5]{vas}}, see also {\cite[\S IX.5]{nik86}} or {\cite[II.C.3.3.5(c) (ii), (iv)]{nik02}}, 
$\theta=\prod_n\theta_n$, where $\{\theta_n\}_n$ satisfies the generalized  Carleson condition, and 
$\theta_n$  is a finite Blashke product with $\deg\theta_n\leq N$ for every $n$, where $\deg \vartheta$ is the quantity of zeros of a finite 
Blaschke product $\vartheta$ accounted with their multiplicities. 

Let $T$ be an a.c. polynomially bounded operator such that $\theta(T)=\mathbb O$. Set $T_n=T|_{\ker\theta_n(T)}$. 
Since $\{\theta_n\}_n$ satisfies the generalized  Carleson condition,  by Lemma \ref{lem3.1} we obtain that 
\begin{equation}\label{3.1} T\approx\oplus_n T_n. \end{equation} 

Let $n$ be fixed. Since $\theta_n(T_n)=\mathbb O$ and $\deg\theta_n\leq N$, there exists a polynomial $p_n$ such that 
$p_n(T_n)=\mathbb O$ 
and $\deg p_n\leq N$. By {\cite[Theorem 2]{bour}}, there exist a constant $C>0$ which depends on $N$ only and an invertible operator $X_n$ such that 
\begin{equation}\label{3.2}\|X_n\|\|X_n^{-1}\|\leq C \ \ \text{ and } \ \ R_n:=X_nT_nX_n^{-1} \ \text{ is a contraction.} \end{equation} 
Note that {\cite[Theorem 2]{bour}} is formulated for an operator 
$T\in\mathcal B(\mathcal H)$, where $\mathcal H$ is a Hilbert space with $\dim\mathcal H\leq N$, but the condition used in the proof is that 
there exists a polynomial $p$ such that $p(T)=\mathbb O$ and $\deg p\leq N$. Indeed, if $T\in\mathcal B(\mathcal H)$ with $\dim\mathcal H\leq N$,
then there exists a polynomial $p$ such that $p(T)=\mathbb O$ and $\deg p\leq N$, but the converse is not true. Thus, {\cite[Theorem 2]{bour}}
is formulated in the weaker form that it is proved. 
 
The conclusion of the theorem follows from \eqref{3.1} and \eqref{3.2}. \end{proof}

\begin{proposition}\label{prop3.3} Assume that inner functions  $\theta_0$ and $\theta_1$ satisfy \eqref{1.4}. 
Then $\theta_0\theta_1$ satisfies \eqref{1.4}.\end{proposition}

\begin{proof} Let $T\in\mathcal B(\mathcal H)$ be an a.c. polynomially bounded operator such that $(\theta_0\theta_1)(T)=\mathbb O$. 
Put $\mathcal H_0=\ker\theta_0(T)$, $\mathcal H_1=\mathcal H\ominus\mathcal H_0$,  $T_0=T|_{\mathcal H_0}$, and $T_1=P_{\mathcal H_1}T|_{\mathcal H_1}$.
 Clearly, $\theta_0(T_0)=\mathbb O$ and $\theta_1(T_1)=\mathbb O$. By \eqref{1.4}, $T_1$ is similar to a contraction. By Theorem 1.2, 
$T$ is similar to a contraction. \end{proof}

\begin{remark}\label{rem3.4} Proposition \ref{prop3.3} gives another proof of Theorem \ref{th3.2}. 
Indeed, if $B=B_1\cdots B_N$, where $B_n$ are 
Blaschke products with simple zeros satisfying the Carleson condition for $1\leq n\leq N$, and $N<\infty$, then $B$ 
satisfies the condition \eqref{1.4}, because $B_n$ satisfy the condition \eqref{1.4}. This proof does not use the results from \cite{bour} 
 and \cite{vas}, 
but this proof gives no estimates of the norm of an operator which realizes the similarity of a polynomially bounded operator 
to a contraction, while \cite{bour}  and \cite{vas}  contain estimates of these norms. \end{remark}

Author does not know another examples of functions satisfying the condition \eqref{1.4} except functions from the condition of 
Theorem  \ref{th3.2}.

\section{Examples.}

In this section, we consider the attempt of straightforward construction of operator which intertwines a polynomially bounded 
operator from Theorem \ref{th2.3} with a contraction, see Remark \ref{rem2.4}. Also, using the example by Le Merdy \cite{lemerdy}, 
we show that the results for 
polynomially bounded operators of the present paper can not be extended to power bounded operators. 

\subsection{On construction of intertwining operators.}

\begin{example}\label{ex4.1}  
 Assume that $\theta$ is an inner function satisfying \eqref{1.4}, and 
$V\in \mathcal B(\mathcal K)$ is an a.c. isometry. Set $\mathcal K_1=\theta(V)\mathcal K$,  $\mathcal H_1=\mathcal K\ominus\mathcal K_1$, 
$V_1=V|_{\mathcal K_1}$, $T_1=P_{\mathcal H_1}V|_{\mathcal H_1}$, $K=P_{\mathcal K_1}V|_{\mathcal H_1}$. 
Assume that $T_0\in \mathcal B(\mathcal H_0)$ and $A\in\mathcal B(\mathcal H_1,\mathcal H_0)$ are such that  
 $$R=\begin{pmatrix} T_0 & A\\  \mathbb O & T_1 \end{pmatrix}$$ is polynomially bounded and $\theta(R)=\mathbb O$.
Define $R_0$ by the formula \eqref{2.5}. By Proposition \ref{prop2.5} (1), $R_0$ is polynomially bounded. Note that $R_0$ has the form  
$$R_0=\begin{pmatrix} T_0 & \ast \\  \mathbb O & V \end{pmatrix}.$$
 By Theorem \ref{th2.3},  $R_0$ is similar to a contraction. 
Let $Y_1\in\mathcal B(\mathcal K,\mathcal H_0)$ be an operator from Remark \ref{rem2.4} applied to $R_0$. Taking into account that 
$\theta(R)=\mathbb O$ and \eqref{2.4} (for $R_0$ and $Y_1$), it is easy to conclude that 
$Y_1|_{\mathcal K_1}=\mathbb O$. Set $Y=Y_1|_{\mathcal H_1}$. 
Then 
\begin{equation}\label{4.1} A=T_0Y-YT_1, \ \text{ and } \ \ \begin{pmatrix} I & Y\\  \mathbb O & I \end{pmatrix} R =\begin{pmatrix} T_0 & \mathbb O\\  \mathbb O & T_1 
\end{pmatrix} \begin{pmatrix} I & Y\\  \mathbb O & I \end{pmatrix}. \end{equation} 

Now suppose that $\theta=B=\prod_nb_{\lambda_n}$ is a Blaschke product with simple zeros $\{\lambda_n\}_n$ 
(here $b_\lambda(z)= \frac{|\lambda|}{\lambda}\frac{\lambda-z}{1-\overline\lambda z}$, $z$, $\lambda\in\mathbb D$), and
$V$ is the unilateral shift of multiplicity 1, that is, the operator of multiplication by the independent variable 
on the Hardy space $H^2(\mathbb D)$. Set $B_n=\prod_{k\neq n}b_{\lambda_k}$ and 
$k_n(z)=\frac{(1-|\lambda_n|^2)^{1/2}}{1-\overline\lambda_n z}B_n(z)$, $z\in \mathbb D$. Then $\mathcal H_1=\vee_n k_n$, 
$\|k_n\|=1$  and $T_1k_n=\lambda_n k_n$. Furthermore, assume that $\{e_n\}_n$ is an orthonormal basis in $\mathcal H_0$, and 
$T_0 e_n=\lambda_n e_n$. Set $a_{jn}=(Ak_n, e_j)$ for all indices $n$, $j$. For $\varphi\in H^\infty$ put 
$A_\varphi=P_{\mathcal H_0}\varphi(R)|_{\mathcal H_1}$. It is easy to see that 
\begin{equation}\label{4.2} (A_\varphi k_n, e_j)=\frac{\varphi(\lambda_n)-\varphi(\lambda_j)}{\lambda_n-\lambda_j}a_{jn}, \ \ j\neq n, \ \ \ \
(A_\varphi k_n, e_n)=\varphi^\prime(\lambda_n)a_{nn}. \end{equation} 
Since $A_B=\mathbb O$ and $B^\prime(\lambda_n)\neq 0$ for every $n$, we conclude from (4.2) that $a_{nn}=0$ for every $n$. 
Therefore, 
\begin{equation}\label{4.3} (A_\varphi k_n, e_n)=0 \ \ \text{  for every } \varphi\in H^\infty.  \end{equation} 
Also, 
\begin{equation}\label{4.4} R(A_{B_n/B_n(\lambda_n)}k_n\oplus k_n)=\lambda_n(A_{B_n/B_n(\lambda_n)}k_n\oplus k_n). \end{equation} 
Let $Y\in\mathcal B(\mathcal H_1,\mathcal H_0)$ satisfy (4.1). Then 
\begin{equation}\label{4.5} A_\varphi=\varphi(T_0)Y-Y\varphi(T_1)\ \  \text{ for every } \ 
\varphi\in H^\infty.  \end{equation} 
Set $\alpha_n=(Yk_n,e_n)$. It is easy to see from \eqref{4.3} and \eqref{4.5} applied with $\varphi=B_n/B_n(\lambda_n)$ that   
\begin{equation}\label{4.6} Yk_n=-A_{B_n/B_n(\lambda_n)}k_n+\alpha_n e_n  \ \ \text{ for every } n. \end{equation} 

Now suppose that $\{\lambda_n\}_n$ satisfies the Carleson condition  (the definition is recalled in Sec. 3).  By Lemma \ref{lem3.1},  
  $\{k_n\}_n$ is equivalent to an orthonormal basis of $\mathcal H_1$, that is, there exists  an inverible operator 
  $X\in\mathcal B(\mathcal H_1)$   such that $\{Xk_n\}_n$ is an  orthonormal basis of $\mathcal H_1$.
Also, by Lemma \ref{lem3.1}, the sequence of spaces $$\{e_n\vee(A_{B_n/B_n(\lambda_n)}k_n\oplus k_n)\}_n$$ 
is equivalent to an orthogonal sequence 
of spaces, because these spaces  are the eigenspaces of a polynomially bounded 
operator $R$ 
with $B(R)=\mathbb O$. Moreover, by \eqref{4.3},  $(e_n, A_{B_n/B_n(\lambda_n)}k_n\oplus k_n)=0$, and 
$$1\leq\|A_{B_n/B_n(\lambda_n)}k_n\oplus k_n\|\leq (M^2/\delta^2+1)^{1/2},$$ where $M$ is 
a constant from 
the condition on polynomially boundedness of $R$ and $\delta=\inf_n|B_n(\lambda_n)|$. 
Thus, the union of the sequences $\{e_n\}_n$ and 
$\{ A_{B_n/B_n(\lambda_n)}k_n\oplus k_n\}_n$ is  equivalent to an orthonormal basis of $\mathcal H_0\oplus\mathcal H_1$. 
Therefore, the operator $Z$ acting by the formula 
$$Ze_n=e_n, \ \ \  Zk_n=A_{B_n/B_n(\lambda_n)}k_n\oplus k_n$$ 
for all indices $n$ is bounded. Clearly, $P_{\mathcal H_0}Z|_{\mathcal H_1}$ is bounded, 
too, and 
$$P_{\mathcal H_0}Zk_n=A_{B_n/B_n(\lambda_n)}k_n.$$ We conclude that the operator $Y$ defined by \eqref{4.6} is bounded for every bounded 
sequence $\{\alpha_n\}_n$. 

 If $B$ is a Carleson--Newman product with simple zeros,  
but the sequence of zeros of $B$ does not satisfy the Carleson condition, the author don't know for such sequences $\{\alpha_n\}_n$ the operator 
$Y$ defined by \eqref{4.6} is bounded. But by Theorems \ref{th2.3} and \ref{th3.2}, such sequence $\{\alpha_n\}_n$ exists. \end{example}

\subsection{On power bounded operators}

Particular cases of Lemma \ref{lem3.1}  and Theorem \ref{th1.2} can be formulated as follows.

\begin{proposition}\label{prop4.2} Assume that $\mathcal H$ is a Hilbert space, $\{x_n\}_n\subset\mathcal H$, $\mathcal H=\vee_n x_n$, 
$\{\lambda_n\}_n\subset\mathbb D$ is such that $\lambda_n\neq \lambda_k$ for $n\neq k$,   
$\{\lambda_n\}_n\subset\mathbb D$ satisfies the Carleson condition,
and $T\in \mathcal B(\mathcal H)$ acts by the formula $T x_n= \lambda_n x_n$ for all indices $n$. 
Assume that  $\mathcal E$ is a Hilbert space, $\{e_n\}_n$ is an orthonormal basis of $\mathcal E$, and $D\in \mathcal B(\mathcal E)$ acts 
by the formula $D e_n= \lambda_n e_n$ for all indices $n$. If $T$ is polynomially bounded, then $T\approx D$. \end{proposition}

\begin{proposition}\label{prop4.3} Assume that  $\mathcal E$ is a Hilbert space with an orthonormal basis $\{e_n\}_n$, 
$\{\lambda_n\}_n\subset\mathbb D$ is such that $\lambda_n\neq \lambda_k$ for $n\neq k$,   
$\{\lambda_n\}_n\subset\mathbb D$ satisfies the Carleson condition,  
and $D\in \mathcal B(\mathcal E)$ acts by the formula $D e_n= \lambda_n e_n$ for all indices $n$. Assume that $\mathcal H$ is a Hilbert space, 
 $T\in \mathcal B(\mathcal H)$ is a contraction, and  $A\in \mathcal B(\mathcal H,\mathcal E)$. 
 Set $$R=\begin{pmatrix} D & A\\  \mathbb O & T \end{pmatrix}. $$  
If $R$ is polynomially bounded, then $R$ is similar to a contraction. \end{proposition}

Example from {\cite[Proposition 5.2]{lemerdy}} shows that Propositions \ref{prop4.2} and \ref{prop4.3} can not be extended 
 to operators  satisfying the
Tadmor--Ritt condition, that is, to operators $T$ for which there exists $C>0$ such that 
$\|(T-zI)^{-1}\|\leq C/|z-1|$ for $z\in\mathbb C$, $|z|>1$ (see also \cite{vitsejfa} for simpler proof). The 
Tadmor--Ritt condition implies power boundedness (see \cite{lyubich}, \cite{nagyzem}, \cite{vitsepal}). 
Consequently, Propositions \ref{prop4.2} and \ref{prop4.3} can not be extended to  power bounded operators.
We sketch the construction of Le Merdy's example.

Assume that  $\mathcal E$ is a Hilbert space with an orthonormal basis $\{e_n\}_{n=0}^\infty$. There exists a family 
$\{\alpha_{kn}\}_{k,n=0}^\infty\subset\mathbb C$ 
such that the sequence $\{x_n\}_{n=0}^\infty$  defined by the formula 
\begin{equation}\label{4.7} x_{2n}=e_{2n}, \ \ x_{2n+1}=e_{2n+1}+\sum_{k=0}^\infty\alpha_{kn}e_{2k}, \ \ \ n\geq 0, \end{equation} 
has the following properties:
\begin{equation}\label{4.8} \|x_n\|\asymp 1, \ \ \sup_{n\geq 0}\|\mathcal P_n\|<\infty, \end{equation}
where $\mathcal P_n\in \mathcal B(\mathcal E)$ acts by the formula $\mathcal P_n x_k=x_k$, $k\leq n$, $\mathcal P_n x_k=0$, $k>n$, 
but $\{x_n\}_{n=0}^\infty$ is not a unconditional basis of $\mathcal E$. Assume that $\{\lambda_n\}_{n=0}^\infty\subset(0,1)$ and 
$\lambda_n<\lambda_{n+1}$, $n\geq 0$. Define $T\in \mathcal B(\mathcal E)$ by the formula $T x_n= \lambda_n x_n$ for all $n\geq 0$.
By {\cite[Lemma 2.2]{vitsejfa}}, $T$ satisfies the Tadmor--Ritt condition.   Consequenly, $T$ is power bounded 
(\cite{lyubich}, \cite{nagyzem}, \cite{vitsepal}). 
Set $$\mathcal E_0=\bigvee_{n\geq 0}e_{2n}, \ \ \mathcal E_1=\bigvee_{n\geq 0}e_{2n+1},$$
$$ D_0 e_{2n}=\lambda_{2n}e_{2n}, \ \ \  D_1 e_{2n+1}=\lambda_{2n+1}e_{2n+1},  \ \ \ n\geq 0.$$
It is easy to see that $T$ has the form 
\begin{equation}\label{4.9} T =\begin{pmatrix} D_0 & A \\  \mathbb O & D_1 \end{pmatrix}, \end{equation}
where $A\in \mathcal B(\mathcal E_1,\mathcal E_0)$ is an appropriate operator. 

Now suppose that $\{\lambda_n\}_{n=0}^\infty$ satisfies the Carleson condition.  If one assume that $T$ is similar to 
a contraction, 
 then the relation $T\approx D_0\oplus D_1$ must be fulfilled. But this relation means that $\{x_n\}_{n=0}^\infty$ is 
 a unconditional basis of $\mathcal E$, a contradiction. 
 Thus, $T$ has a form as in Proposition  \ref{prop4.2}, but $T$ does not satisfy the conclusion of Proposition  \ref{prop4.2}. 
 Also, since $T$ has the form \eqref{4.9}, $T$ has a form as in Proposition \ref{prop4.3}, but $T$ does not satisfy the conclusion of 
 Proposition  \ref{prop4.3}. 

\begin{remark}\label{rem4.4} Example of the  family $\{\alpha_{kn}\}_{k,n=0}^\infty$ such that the sequence $\{x_n\}_{n=0}^\infty$ 
defined by  \eqref{4.7} 
has the needed properties can be found in {\cite[Example III.14.5, p. 429]{singer}}. Namely, assume that the sequence $\{a_n\}_{n=0}^\infty$ 
has the following properties:
$$ a_n\geq 0, \ \ \sum_{n=0}^\infty n a_n^2<\infty, \ \text{ and } \  \sum_{n=0}^\infty  a_n =\infty.$$
Set 
\begin{equation}\label{4.10} \alpha_{kn}=0, \ \text{ if }  k<n, \ \ \text{ and } \ \alpha_{kn}=a_{k-n}, \ \text{ if }  k\geq n.  \end{equation}
Then the family $\{\alpha_{kn}\}_{k,n=0}^\infty$ has the needed properties. 

Suppose that $\{a_n\}_{n=0}^\infty$ is the sequence 
(of not necessary nonnegative) complex numbers such that $\sum_{n=0}^\infty |a_n|^2<\infty$, 
$\{\alpha_{kn}\}_{k,n=0}^\infty$ is defined by \eqref{4.10}, and $\{x_n\}_{n=0}^\infty$ is defined by  \eqref{4.7}. Set 
\begin{equation}\label{4.11} \varphi(z)=\sum_{n=0}^\infty a_n z^n,  \ \ \ z\in\mathbb D.  \end{equation}
 It is easy to see that  \eqref{4.8} is fulfilled if and only if the Hankel operator 
with the symbol $\psi$,  $\psi(z)=\overline z \varphi(\overline z)$, $z\in \mathbb T $ (here $\mathbb T$ is the unit circle),  is bounded, that is, $\varphi\in BMOA$, while $\{x_n\}_{n=0}^\infty$ is a unconditional basis 
if and only if the Toeplitz operator with the symbol $\varphi$ is bounded, that is, $\varphi\in H^\infty$ 
(on Hankel and Toeplitz operators and their symbols see, for example, {\cite[Ch. 1.1, 1.3]{peller}}, 
see also {\cite[\S VIII.2, \S VIII.6, Appendix]{nik86}}, {\cite[Ch. I.B.1, Sec. I.B.4.1]{nik02}}). 
The condition $\sum_{n=0}^\infty n |a_n|^2<\infty$ 
is a simple sufficient condition for the boundedness of the Hankel operator, 
while the condition $\sum_{n=0}^\infty  a_n <\infty$ 
is necessary and sufficient condition for $\varphi\in H^\infty$, where $\varphi$ is from  \eqref{4.11} with 
 $a_n\geq 0$ for all $n\geq 0$. \end{remark}

\end{document}